\documentclass[aap]{imsart}

\RequirePackage{amsthm,amsmath,amsfonts,amssymb}
\RequirePackage[numbers]{natbib}
\RequirePackage[colorlinks,citecolor=blue,urlcolor=blue]{hyperref}
\RequirePackage{graphicx}

\startlocaldefs
\theoremstyle{plain}
\newtheorem{theorem}{Theorem}[section]
\newtheorem{lemma}[theorem]{Lemma}
\newtheorem{prop}[theorem]{Proposition}
\newtheorem{cor}[theorem]{Corollary}
\theoremstyle{remark}
\newtheorem{definition}[theorem]{Definition}
\newtheorem{assump}{Assumption}

\newtheorem*{rmk}{Remark}

\endlocaldefs

\newcommand{\Z}{\ensuremath{\mathbb Z}} 
\newcommand{\diag}{\operatorname{diag}}

\DeclareGraphicsExtensions{.pdf,.png,.jpg,.eps}

\begin{document}


%
%

\begin{frontmatter}
\title{Parameter Identifiability of a Multitype Pure-Birth Model of Speciation}
\runtitle{Parameter Identifiability of a Model of Speciation}

\begin{aug}
\author[A]{\fnms{Dakota}~\snm{Dragomir}\ead[label=e1]{dbdragomir@alaska.edu}},
\author[A]{\fnms{Elizabeth S.}~\snm{Allman}\ead[label=e2]{e.allman@alaska.edu}}
\and
\author[A]{\fnms{John A.}~\snm{Rhodes}\ead[label=e3]{j.rhodes@alaska.edu}}
\address[A]{Department of Mathematics and Statistics,
University of Alaska Fairbanks, Fairbanks, AK 99709\printead[presep={,\ }]{e1,e2,e3}}

\end{aug}

\begin{abstract} 
Diversification models describe the random growth of evolutionary trees, modeling the historical relationships of species through speciation and extinction events. One class of such models allows for independently changing traits, or types, of the species within the tree, upon which speciation and extinction rates depend. Although identifiability of parameters is necessary to justify parameter estimation with a model, it has not been formally established for these models, despite their adoption for inference. This work establishes generic identifiability up to label swapping for the parameters of one of the simpler forms of such a model, a multitype pure birth model of speciation, from an asymptotic distribution derived from a single tree observation as its depth goes to infinity. Crucially for applications to available data, no observation of types is needed at any internal points in the tree, nor even at the leaves.
\end{abstract}

\begin{keyword}[class=MSC]
\kwd[Primary ]{60J80}
\kwd{92D15}
\kwd[; secondary ]{60J85}
\end{keyword}

\begin{keyword}
\kwd{Diversification model}
\kwd{Multitype branching process}
\end{keyword}

\end{frontmatter}


\section{Introduction}

Species diversification models are used in Biology to make inferences about historical speciation and extinction rates over the time since a group of species, or \emph{taxa}, evolved from a common ancestor.
By providing information on rates of speciation and extinction, inference with these models seeks to give insight into the evolutionary dynamics leading to the present diversity of life.
These models have a long history, starting with Yule's constant-rate pure-birth model \cite{Yule1924}, and a fairly large literature has developed.

Diversification models describe a process beginning with  a single lineage at some time in the past, which as time progresses may speciate or go extinct. When a speciation occurs the edge bifurcates into two edges, with the number of lineages increasing by 1. When an extinction occurs, the lineage ends, and the number of lineages decreases by 1. After either event, the process continues forward, independently on all  lineages, producing a growing tree structure until the present time is reached. This tree, which has both topological and metric structure, constitutes an observation.
(In applications, it may be necessary to consider the \emph{reconstructed tree}, which is obtained by removing all tree edges with no descendents at the present \cite{Nee1994,Harvey1994}.)

Two basic sorts of these models have found common use in empirical studies. In the first, the speciation and extinction rates are functions of time, and apply to all taxon lineages present at any moment. This can be thought of as modeling exogenous factors, such as environmental conditions, that affect all taxa in the tree identically. Since all lineages behave in the same probabilistic way at any moment, it is not hard to show that the exact branching pattern of the tree-structure is irrelevant, with all the information in a tree observation being captured by the number of lineages as a function of time.  Thus the work  on \emph{time-dependent birth-death models}  by  Kendell is foundational \cite{Kendall1948}.

In the second sort of diversification model, which we call the \emph{multitype birth-death tree model}, lineages are assigned one of a finite number of \emph{types} at each moment, with the model's speciation and extinction rates dependent only on the type. Over time, however, species may change types at fixed switching rates. This models endogenous factors, such as a particular biological trait a taxon may possess, including, for instance, a morphological feature, behavior, or whether a particular gene is present and active in an organism.  A given type might correlate with faster or slower speciation than another, and/or affect the extinction rate. For these models the branching structure of a tree observation does matter, as taxa present at a given time may each have different types, and thus different tendencies to speciate or go extinct. 

The Binary State-specific Speciation and Extinction (BiSSE) model  formalized the multitype framework for biological applications \cite{Maddison2007}. Multitype (MuSSE) and quantitative-type (QuaSSE) variants of the model were subsequently proposed \cite{FitzJohn2012}. Although these works assumed the type is observed for the extant taxa at the leaves of a tree, we consider the multitype birth-death tree model with no type information observable for any lineage at any time, as type observations are unnecessary for our results. Indeed, the usefulness of these models  to infer correlation between observed types and diversification rates from data with type information for extant taxa has been called into question \cite{RaboskyGoldberg2015}.

Many other diversification models have been proposed, combining or extending these basic frameworks, with \cite{Stadler2013} offering one review. New variants continue to be developed, e.g.,  \cite{Maliet2019,Stadler2019,RasStadler2019,Barido-Sottani2020}. 

\medskip

When these models are used for inference,  the data is taken as a single tree assumed to show the true evolutionary relationships of the taxa. (In practice, this tree itself must be inferred, usually from sequence data using phylogenetic  and/or phylogenomic methods which we do not discuss here.) Multiple trees which one can reasonably hypothesize were generated with the same parameter values are simply not available. If the tree is sufficiently large, researchers hope it provides enough information to infer the speciation and extinction parameters reasonably well. More precisely, it has been implicitly assumed that the inference is statistically consistent, in the sense that as the number of taxa increases toward infinity (i.e., the tree grows larger), the probability of inferring model parameters arbitrarily close to the generating ones approaches 1. Establishing such a result, however, requires showing identifiability of the model parameters: A distribution derived from an observation of a single tree has a limit, as the number of taxa approaches infinity, that uniquely determines all parameter values. 

Of course a full proof of the statistical consistency of a particular estimator requires additional arguments. For instance,  the standard results on
the consistency of maximum likelihood
assume the availability of multiple independent samples, and therefore cannot be applied. Leroux's result on the  consistency of maximum likelihood inference from a single sequence of observation from a Hidden Markov Model \cite{Leroux1992} is analogous to what is need for applications of these diversification models. Nonetheless, establishing parameter identifiability is the first step toward this goal.

Recent work has shown that the first type of diversification model, with time-dependent rates, does not in fact have identifiable parameters, despite its widespread use by empiricists \cite{Louca2020}. This result, which holds even if one allows for identification to be based on arbitrarily many independent tree observations with the same underlying rate parameters, was compellingly illustrated by construction of examples of wildly different rate functions producing identical tree distributions. An instance of this lack of identifiability had in fact appeared earlier, in an argument in which speciation rates were modified and extinction rates set to zero without changing the model distribution \cite{Nee1994}.

Little work, however, has addressed identifiability questions for multitype birth-death tree models. 
The strongest results on parameter identifiability for a pure birth model focus on a tree's topological features but assume the types of both leaf nodes and their parents are observed \cite{Popovic2016}. In biological applications, however, the type of a leaf of the tree may be observable,  but the type of the parent nodes is virtually never known. Thus no identifiability result relevant to typical data analyses has been produced.

One might hope that the analysis of multitype birth-death tree models would be simpler than for a time-dependent rate model, as its parameter space is finite dimensional. On the other hand, while trees produced by the time-dependent rate models can be summarized by the counts of lineages through time with no loss of information, this is not true for the multitype models. Effectively extracting information from a tree with both topological and metric structure requires a new approach.

\medskip

In this paper, we investigate parameter identifiability of the \emph{multiype pure-birth tree (MPBT) model} with any finite number of types. We thus restrict extinction rates for all classes to be zero. This model has also been called the  multitype Yule model \cite{Popovic2016}. We assume only that the metric tree is observable, with no information on the types either at points internal to the tree or at the leaves. More formally, we establish \emph{generic identifiability of parameters up to label swapping}. ``Generic" means the result holds if we exclude parameters lying in a measure-zero subset of the parameter space. We give an explicit characterization of such a measure-zero exceptional set, as the zero set of a certain polynomial. ``Up to label swapping" means that there are certain symmetries of the parameter space, arising from interchanging types, so that their corresponding speciation and switching rates are also interchanged, that have no effect on the model's behavior.  Generic identifiability up to label swapping is often the strongest form of identifiability that holds in models with hidden variables \cite{Allman2009}, and since we treat the types as unobservable, its appearance here is not surprising.

Our arguments draw on several earlier works. The first is \cite{Athreya1968} on Multitype Continuous Time Markov Branching Processes. In fact, these models and the MBPT model have the same underlying structure. But much of the classical branching process literature allows only for observing type counts over time, and not for observing the tree structure indicating the branching of specific lineages.  The  MPBT model, in contrast, treats the tree structure as observable, with type information hidden. Thus while providing an important tool in this work, the results of \cite{Athreya1968} are not immediately applicable to the MPBT model.

The second result crucial to our work is a general theorem on identifiability up to label swapping of parameters of a mixture model of product distributions \cite{Allman2009}.
In applying this to the MPBT model, we consider the joint distribution of edge lengths around a node on a uniformly-at-random chosen edge of a random tree, as the random tree grows arbitrarily large. Due to conditional independence of edge lengths conditioned on the type of the shared node, this joint distribution takes the form of a mixture distribution (over types) of product distributions. Although additional work is necessary to show parameter identifiability, this theorem is a crucial ingredient in our argument.

Although we do not address the multitype birth-death tree model with non-zero extinction rates here, we believe that our approach provides a pathway toward a more general result. 

\medskip
 
This paper is structured as follows. In Section \ref{sec:overview} we provide a more formal definition of the MPBT model, and begin its analysis by deriving formulas related to the generation of a single edge in the tree in Section \ref{sec:edge}. Section \ref{sec:typecount} uses the results in \cite{Athreya1968} to obtain asymptotic results on the distribution of types across lineages in the tree at times increasingly distant from the root of the tree. Then, in Section \ref{sec:main}, we bring these ingredients together, and apply the theorem of \cite{Allman2009} to obtain our main results. Concluding remarks appear in Section \ref{sec:discuss}.


\section{Model definition}\label{sec:overview}

In this section we formalize the Multitype Pure-Birth Tree model, in a form useful for our analysis.

\smallskip
Let $m$ be a positive integer denoting the number of types, and denote the set of types by $[m]=\{1,2,\dots,m\}$.

The parameter space of the MBDT model with $m$ types is all 3-tuples $(\boldsymbol \pi, \boldsymbol \lambda, S)$ described as follows:

A \emph{root distribution} $\boldsymbol \pi=(\pi_1,\pi_2,\dots ,\pi_m)$, with $\pi_i\ge 0$, $\sum_i \pi_i=1$
gives probabilities $\pi_i$ of type $i$ being chosen for the tree root.  A vector $\boldsymbol \lambda=(\lambda_1,\lambda _2,\dots, \lambda_m)$ with non-negative entries gives \emph{speciation rates} $\lambda_i$ for type $i$. An $m\times m$ matrix
$S=(s_{ij})$ with non-negative off-diagonal entries and rows summing to 0 gives scalar \emph{type switching rates} $s_{ij}$ from type $i$ to type $j$, $i\ne j$. Note that $S$ is determined by the $m^2-m$ independent scalar switching rates.

 \subsection{The edge process model} We first describe how an edge of a tree is produced under the model. 
 As edges of the tree are produced independently conditioned on their starting types, a description of  a single edge is sufficient.
 
 We view an edge as growing with time, randomly changing the type of its leading point as it does so. At any time the edge may speciate, at a rate $\lambda_i$ determined by its current type $i$. When speciation occurs, the edge ceases to grow, and in the full model two new edge processes are started for its descendent edges. However,  in formalizing the edge process we describe the speciation of an edge as the process entering an absorbing state, for mathematical convenience.

For each type  $i\in [m]$, define two states $i_-, i_+$.
At any time, state $i_-$ indicates that the current leading point of the edge has type $i$ and that the edge has not yet speciated. The absorbing state $i_+$ represents that a speciation has occurred and at the time of speciation the leading point had type $i$.
The parameter $s_{ij}$, $i\ne j$, is thus a rate of change from state $i_-$ to state $j_-$, while $\lambda_i$ is the  rate of change from state $i_-$ to $i_+$. No other instantaneous state changes are allowed.

\begin{definition} 
\label{def1} The \emph{$m$-type pure-birth edge process} $E_{\tau}= E_\tau(\widetilde{\boldsymbol \pi},\boldsymbol \lambda, S)$ with $\widetilde \pi_i \ge0 $, $\sum_i \widetilde \pi_i=1$, is the $2m$-state continuous-time Markov process over $\tau \in[0, \infty)$ with 
states $$1_-,2_-,\dots, m_-,1_+,2_+,\dots, m_+,$$ initial state distribution $(\widetilde {\boldsymbol \pi},\mathbf 0)\in \mathbb R^{2m}$,  and $2m\times 2m$ transition rate matrix
$$Q := 
\begin{bmatrix}
S-\diag(\boldsymbol \lambda) & \diag(\boldsymbol \lambda) \\
\mathbf 0 & \mathbf 0 
\end{bmatrix},$$ 
where the rows and columns of $Q$ are ordered by states as above.  Here $\mathbf 0$ is a vector or matrix of 0s, and $\diag(\boldsymbol \lambda)$ is the diagonal matrix formed from vector $\boldsymbol \lambda$.

The transition probability matrix associated to $E_{\tau}$ is  $$P(\tau) = \exp({Q\tau}),$$ with $P_{ij}(\tau)$ giving the probability that an edge is in state $j$ at time $\tau$ given that it was in state $i$ at time $0$.

\end{definition}

\begin{definition}
The  $\emph{speciation time}$ $\mathcal T$ associated to $E_{\tau}(\widetilde{\boldsymbol \pi},\boldsymbol \lambda, S)$ is the $[0, \infty]$-valued  random variable
$$\mathcal T = \inf \left(  \{\tau \ge 0 \mid E_{\tau} \in \{1_+, 2_+,\dots, m_+\}\} \cup \infty\right).$$ 
\end{definition}

A realization of the edge process that reaches a ``+'' state is viewed as an edge of length $\mathcal T$, the time at which a speciation occurs.  Each point (time $\tau$) along the edge is ``colored" by type $i$ if the process is  in state $i_-$ (or state $i_+$ at its endpoint) at that time. 
Under mild assumptions, the edge length is finite with probability 1, as is shown below. Although for the  MPBT model colors on edges are ultimately hidden, they play an important role in our arguments.

The terminal edges of the tree are produced by terminating edge processes at a specific time, before they may have reached an absorbing state. Formally defining such a \emph{truncated edge process} and the colored edge it produces, is straightforward.
 
Due to the time-homogeneous Markov formulation of the edge process, we may equivalently produce an edge either from a single  process reaching a ``$+$'' state, or  by starting the process, truncating it before it enters a ``$+$'' state, starting a new process in the final state of the truncated one, and then conjoining the edges produced. Likewise, to produce an edge from the truncated process, we may allow the process to continue to a later time, and then truncate the edge that was produced to an initial segment.
 
\subsection{The multitype pure-birth tree model}
 
We now define the MPBT model, as a generative model producing a tree. Let $T>0$ be the depth (length of all paths from root to any tip) of the tree to be sampled.

\begin{enumerate}

\item The process begins with a root node.  With parameters $(\boldsymbol \pi,\boldsymbol  \lambda, S),$ generate from an edge process a colored descendent edge from the root to a node of type $i$, the only current tip of the tree. 

If the length of this edge is $\ge T$, truncate it to length $T$, and go to Step \ref{step:4}.

Otherwise, at this node attach two descendent edges of length 0, with points on them colored by $i$. The tree now has 2 tips.

\item \label{step:2} If the tree currently has $k$ tips,  for each tip generate a descendent edge via independent edge processes with parameters $(\mathbf e_i,\boldsymbol  \lambda, S)$, where $i$ is the type of the tip and $\mathbf e_i$ the standard basis vector in $\mathbb R^m$. Truncate all edge processes at the time $\tau$ when the first reaches a ``+" state. The colored edges for each tip are conjoined to the edges (possibly of length 0) leading to the tip. 

If the path length from the root to a tip of the tree is $\ge T$, truncate all terminal edges so that all paths from root to leaves have length $T$, and go to Step \ref{step:4}.

Otherwise, at the tip that arose from reaching state $j_+$, we attach two descendent edges of length 0 with points on them colored by $j$.

\item Go to step \ref{step:2}.

\item \label{step:4} Uncolor all edges to obtain a sampled tree.

\end{enumerate}

An example simulation of a colored tree from a binary-type model is shown in Figure \ref{fig:f1}, with the color hidden in Figure \ref{fig:f4}.

\begin{rmk} Inherent in the model are several notions of time. For an individual edge process, $\tau$ is a time variable, with $\tau=0$ at the parental node in the edge. For the tree generation process overall, we use $t$ as the time variable, with $t=0$ at the root. If the edge process starting at the root enters a ``+'' state at time $\tau=\mathcal T_0$, then that root edge has length $\ell=\mathcal T_0$ and at its child node $t=\mathcal T_0$. Then if the edge process for an edge descending  from the first speciation produces an edge of length $\mathcal T_1$, then at its child node $t=\mathcal T_0+\mathcal T_1$. In general, a point on any edge $e$ at time $\tau$ has 
$$t=\tau+\sum_{\tilde e \text{ above } e} \mathcal T_{\tilde e}.$$   
We can thus view a random tree as growing with time $t$, as its terminal edges lengthen while changing type, and speciate.
\end{rmk}

\begin{rmk}
While we have defined the MPBT model as starting with a single edge descending from the root node, it is equally common to define diversification model starting at a bifurcating root. The modifications to the definition that are necessary to do so are straightforward, and working in that context would have no substantive impact on the arguments which follow.
\end{rmk}

\begin{rmk}
Even if $T\to\infty$, a single observed tree does not allow for the identification of  $\boldsymbol \pi$, so we focus on identifying the pair $(\boldsymbol \lambda, S)$. This factor of the parameter space can be identified with  the non-negative orthant of $\mathbb R^{m^2}$.\end{rmk}

\begin{figure}
\includegraphics[width=4.in]{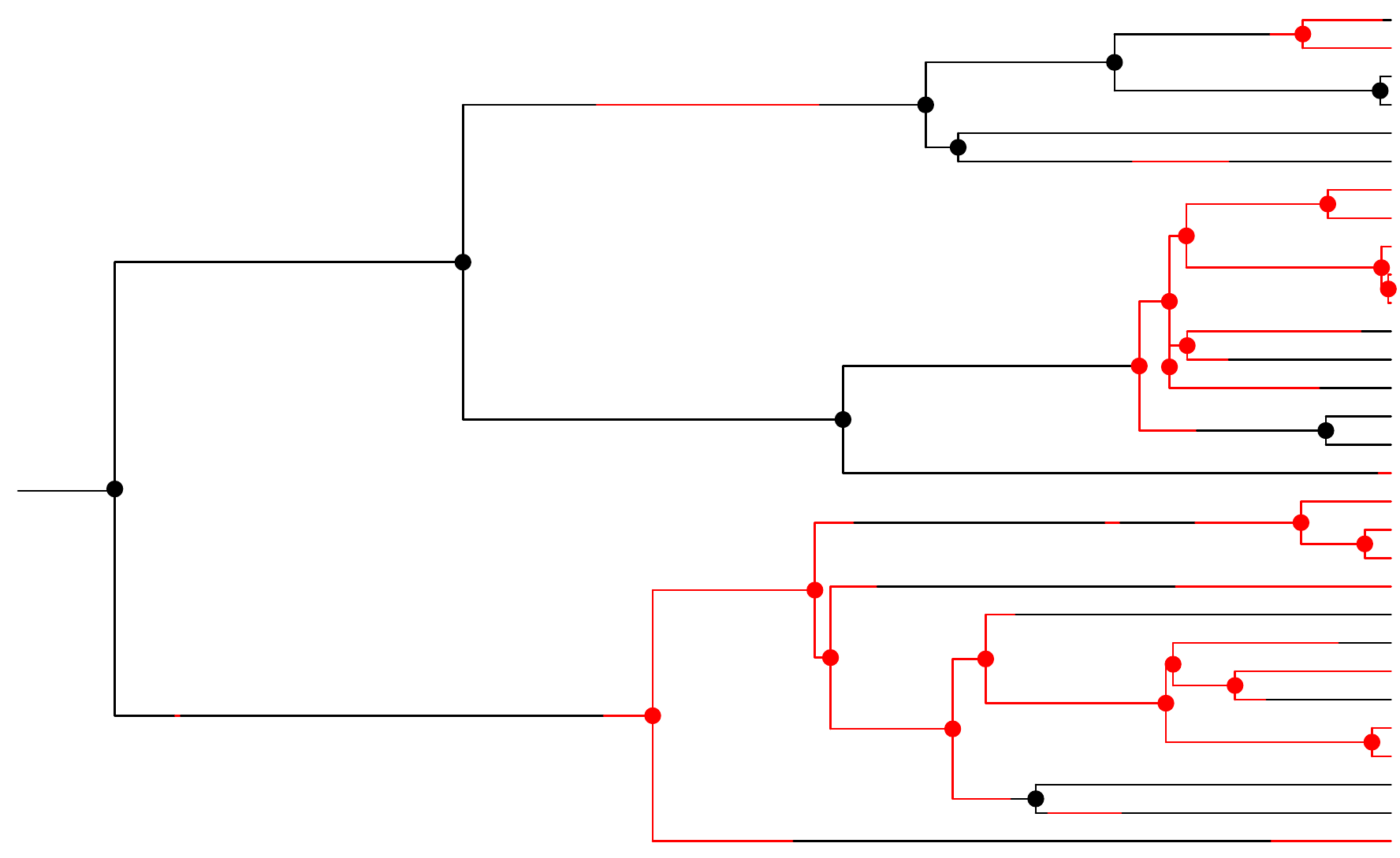}
\caption{A finite-length colored tree generated by the binary-type pure birth tree model, before colors are hidden. Here black represents  type 1 and red type 2, with $\lambda_1=0.1,$  $\lambda_2=0.5,$ $s_{12}=0.1,$ $s_{21}=0.2$. Only the uncolored tree is observed.}\label{fig:f1}
\end{figure}

\section{The edge process}\label{sec:edge}

For parameters $(\boldsymbol \lambda, S)$, let $D=\diag(\boldsymbol \lambda)$ and $U=S-D$, so that the edge process $E_\tau$ has Markov rate matrix 
$$Q =\begin{bmatrix}
U & D\\
\mathbf 0 & \mathbf 0 
\end{bmatrix}.$$

\begin{lemma}\label{lem:edgeTrans}
The transition probability matrix for $E_{\tau}$ is $$P(\tau) = \begin{bmatrix}
\exp (U\tau) & f(U\tau) D \tau\\
\mathbf{0} & I \\
\end{bmatrix},$$
where $f(A)=\sum_{n=0}^\infty  \frac 1{(n+1)!} A^n$ satisfies $f(A)A=\exp(A)-I$. 
\end{lemma}
\begin{proof}
For $n\ge1$ $$Q^{n} = 
\begin{bmatrix}
U^{n} & U^{n - 1}D \\
\mathbf{0} & \mathbf{0} \\
\end{bmatrix},$$ so 
$$P(\tau) = I + \sum_{n = 1}^{\infty}\frac 1{n!} {\begin{bmatrix}
U^{n} & U^{n - 1}D \\
\mathbf{0} & \mathbf{0} \\
\end{bmatrix}\tau^{n}}= \begin{bmatrix}
\sum_{n = 0}^{\infty}\frac 1{n!} {U^{n}\tau^{n}} & \sum_{n = 1}^{\infty}\frac1{n!} {U^{n - 1}D\tau^{n}} \\
\mathbf{0} & I \\
\end{bmatrix}.$$ 
\end{proof}

For technical reasons we impose the following assumption, which is also biologically plausible.
 
\begin{assump} \label{a:lambdapos} The speciation rates $\lambda_i$ are positive for all $i$.
\end{assump}

\begin{lemma}\label{lem:eigs}
Let ($\boldsymbol \lambda$, $S$)  be  parameters for a MPBT edge process satisfying Assumption \ref{a:lambdapos}.
Then $U$ is non-singular and all eigenvalues of $U$ have negative real part. 
\end{lemma}

\begin{proof} 
The assumption implies that $U$ is strictly diagonally dominant, that is, the absolute value of each diagonal entry is strictly greater than the sum of the absolute values of all other entries in its row. Thus $U$ is non-singular \cite{Horn2012}.
Since the diagonal entries are also negative, by the Gershgorin Circle Theorem every eigenvalue of $U$ will have negative real part.    
\end{proof}

\begin{prop}
\label{prop:condT}
Let $F_{i}$ denote the cdf of the speciation time $\mathcal T$ conditioned on $E_{0} = i_{-}$, and $\mathbf 1$ be the vector of $1$s. Then $F_i$ is given by the $i$-th entry of $$\mathbf{1} - \exp({U\tau})\mathbf{1}.$$ Moreover, under Assumption \ref{a:lambdapos}, $\mathcal T$ is finite with probability $1$.
\end{prop}

\begin{proof} Since $\mathcal T$ is the time $E_{\tau}$ first enters any of the absorbing states $j_+$, $F_{i}$ is the sum across the $i_{-}$ row of the upper right $m\times m$ block of $P(\tau)$.
From Lemma \ref{lem:edgeTrans}, using that $D\mathbf{1} = -U\mathbf{1}$,
the column vector of the $F_{i}$s is therefore given by $$f(U\tau)  D\tau \mathbf{1} = -f(U\tau)U\tau\mathbf{1} = \mathbf{1} - \exp({U\tau})\mathbf{1}.$$ 
   
Under Assumption \ref{a:lambdapos}, by Lemma \ref{lem:eigs} the eigenvalues of $U$ have negative real parts, so $\lim_{\tau \to \infty}\exp({U\tau}) = \mathbf{0}$. Thus $\lim_{\tau \to \infty}F_{i}(\tau) = 1$ for each $i$, implying that $\mathcal T$ is finite with probability $1$.
\end{proof}

\begin{prop} 
\label{prop:asycondprob}
Let $P_{i_{-}, j_{+}} = \lim_{\tau \to \infty}P_{i_{-}, j_{+}}(\tau)$ denote the asymptotic probability of transition to $j_{+}$ conditioned on $E_{0} = i_{-}$. Then under Assumption \ref{a:lambdapos}, $P_{i_{-}, j_{+}}$ is the $(i, j)$-entry of $-U^{-1}D$.
\end{prop}

\begin{proof} The matrix $P_{-,+}(\tau)$ with entries $P_{i_{-}, j_{+}}(\tau)$ is the upper right $m \times m$ block of $P(\tau)$, so by Lemma \ref{lem:edgeTrans}, $$ P_{{-}, {+}}(\tau)=f(U\tau)D\tau=(\exp({U\tau}) - I)U^{-1}D,$$
using that $U$ is non-singular by Lemma \ref{lem:eigs}.
But $\lim_{\tau \to \infty}\exp({U\tau}) = \mathbf{0}$ because  $U$'s eigenvalues have negative real parts. Thus 
$$P_{-, +}=\lim_{\tau \to \infty}(\exp({U\tau}) - I)U^{-1}D = (\mathbf{0} - I)U^{-1}D = -U^{-1}D.$$
\end{proof}


\section{Type Counting Process}\label{sec:typecount}

Another ingredient of our approach to establishing the  identifiability of MPBT model parameters is an analysis of an associated classical branching process, in which only the type counts are observed. More specifically,  it records the number of edges of the tree which have each type as a function of time, but retains no information on the topology of the tree. 
We call this the \emph{type counting process}, and in this section use established results to determine the asymptotic behavior of the relative frequencies of each type.

\begin{definition}
\label{def3}
For $i\in[m]$, let ${N}^{i}_{t}$ denote the number of edges in a colored random tree arising from the colored MPBT model that exist at time $t$ and are of type $i$ at that moment. The $\emph{type counting process}$ $N_{t}$ is the $(\Z^{\ge 0})^m$-valued continuous-time stochastic process over $[0, \infty)$ defined by $N_{t} := (N^{1}_{t}, N^{2}_{t},\dots, N_t^m)$.  The $\emph{relative frequency process}$ is $R_{t} = N_t/(\sum_{i=1}^m N_t^i)$, provided
the denominator is non-zero.
\end{definition}

The asymptotics of the relative frequencies follow from results of \cite{Athreya1968} on \emph{multitype} \emph{continuous-time} \emph{Markov} \emph{branching} \emph{processes}, specifically Theorems 1 and 2 of that work, which are paraphrased below as Theorem \ref{thm:A}. Such a model can be described as a process where individuals of type $i$ live an exponentially-distributed length of time (whose rate only depends on type) and on death may be replaced by individuals of any type according to a distribution over $(\Z^{\ge 0})^{m}$. 

To place the type counting process of the MPBT model into this framework, both speciation and change in type are viewed as deaths.  Speciation results in replacement by 2 individuals of the same type, and change in type results in replacement by an individual of a different type. Since a speciation ``death" of a type $i$ individual occurs with rate $\lambda_i$, and a type change ``death" of a type $i$ individual followed by replacement with type $j$ occurs with rate $s_{ij}$, the combined rate of death for type $i$ is $\lambda_i +\sum_{j\ne i} s_{ij}$. When a death occurs, it is a speciation with probability
$$\frac{\lambda_i}{\lambda_i +\sum_{j\ne i} s_{ij}},$$ and a change to type $j$ with probability $$\frac{s_{ij}}{\lambda_i +\sum_{j\ne i} s_{ij}}$$

Basic properties of the type counting process are summarized in the following.
 
 \begin{lemma}
\label{lem2}
The type counting process $N_{t}$ of the MPBT model is a strong Markov, continuous-time, $m$-type branching process, where each type $i$ death has an offspring distribution defined by the multivariable probability generating function  $$h_{i}(x_{1}, x_{2},\dots, x_m) = \frac{\lambda_i}{\lambda_i +\sum_{j\ne i} s_{ij}}x_{i}^{2} + \sum_{j\ne i} \frac{s_{ij}}{\lambda_i +\sum_{j\ne i} s_{ij}}x_{j}.$$ 
\end{lemma}

We introduce yet another matrix defined in terms of the MPBT model parameters, as its leading eigenvalue and corresponding eigenvector plays a large role in the counting process's behavior.

\begin{definition}
\label{def4}
Given parameters $(\boldsymbol \lambda , S)$ of the MPBT model, let 
$$A =  S+D.$$  A \emph{leading eigenvalue of $A$} is an eigenvalue, $\omega$, with the largest real part, and a \emph{normalized leading left eigenvector} of $A$, is a left eigenvector for $\omega$ with $\sum_{i}u_{i} = 1$.
\end{definition}
The matrix $A$ is the infinitesimal generator of the conditional expectation of the $N_i$s. More precisely,
$$\exp(At)=M_t=(m_{ij}(t))$$ with
$$m_{ij}(t)= \mathbb E\left [N^j_t|N_0=\mathbf e_i\right ],$$
where $\mathbf e_{i}$ is the $i$-th standard basis vector.

We will shortly show $\omega$ and $\mathbf u$ are uniquely determined, under an additional assumption.

\begin{assump}\label{a:spos}
The off-diagonal entries of $S$ are positive, i.e., $s_{ij}>0$ for $i\ne j$ .
\end{assump}

\begin{lemma} \label{lem3}
For parameters $(\boldsymbol \lambda , S)$ of the MPBT model satisfying Assumption \ref{a:spos},
\begin{enumerate} 
\item\label{c:1} $M_t=\exp({At})$ has positive entries for $t > 0$.
\item \label{c:2} A has a unique leading eigenvalue $\omega$, which is both simple and real. Moreover the corresponding normalized left eigenvector $\mathbf u$ can be chosen to have all positive components.
\end{enumerate}
\end{lemma}

\begin{proof} 
Fix $t > 0$. Then, using Assumption \ref{a:spos}, $A$ has positive off-diagonal entries, so there is a real $k$ such that $B = At + kI$ has positive entries. Since $B, kI$ commute, it follows that $e^{At} = e^{B - kI} = e^{-k}e^{B}$. Since $B$ has positive entries, $e^{B}$ does as well. Thus, $e^{At}$ has positive entries.

The Perron-Frobenius Theorem applied to $B$ shows it has a unique dominant (i.e., of maximal absolute value) eigenvalue $\omega$ which is also positive and simple, with a unique normalized left eigenvector $\mathbf u$ whose components are all positive. Since $A$ has the same eigenvectors, and eigenvalues shifted by $-k$ and scaled by $1/t$, the second claim follows.
\end{proof}

Key properties of the counting process follow from the following more general theorem on classical branching processes.

\begin{theorem}\cite{Athreya1968}
\label{thm:A}
Let $X_{t}$ be a strong Markov, continuous-time, $m$-type branching process over $[0, \infty)$ which takes values in $(\Z^{\ge 0})^{k}$. Let $M_t=\exp(At) $ be the conditional expectation matrix. Let $h_{i}(x_{1}, ..., x_{k})$ be the  offspring probability generating function for type $i$.    
  
If $M_{t_0}$ has positive entries for some $t_{0} > 0$, and $h_i(s)$ is of degree $>1$ for all $i$, then as $t \to \infty$,
$$X_{t}e^{-\omega t} \xrightarrow{a.s.} W\mathbf u,$$
where $W$ is a non-negative random variable, $\omega$ is the leading eigenvalue of $A$, and $\mathbf u$ is the positive normalized left eigenvector of $A$ associated with $\omega$.

Moreover,  if $\boldsymbol \xi^i =(\xi^i_j)$ are random variables with generating functions $h_i$, then
\begin{equation}\label{eq:logcond}
\mathbb{E}\left[\xi^{i}_j\log(\xi^{i}_j)\right] < \infty
\end{equation}
for all  $i,j$
if and only if for all $i$
 $$\mathbb{P}(W = 0 \ \vert \ X_{0} = \mathbf e_{i}) = \mathbb{P}(X_{t} = 0 \textrm{ for some t} \ \vert \ X_{0} = \mathbf e_{i}).$$
 \end{theorem}

\begin{cor} 
\label{cor4} Consider the counting process associated to the MPBT model for parameters $(\boldsymbol \pi, \boldsymbol\lambda,S)$. Then under Assumptions \ref{a:lambdapos} and \ref{a:spos},
$\sum N^i_t$ is non-zero and
as $t \to \infty$, $$R_{t} \xrightarrow{a.s.} \mathbf u,$$ where $u$ is the positive  normalized leading left eigenvector of $A$.
\end{cor}

\begin{proof} Using the assumptions and Lemmas \ref{lem2} and \ref{lem3}, the hypotheses of Theorem \ref{thm:A} are met, including inequality \eqref{eq:logcond}.
Thus
$$N_{t}e^{-\omega t} \xrightarrow{a.s.} W\mathbf u,$$ where $\omega$ is the leading eigenvalue of $A$, $\mathbf u$ is its positive normalized left eigenvector, and $W$ is a non-negative random variable.    
    
Since the random variable $\sum N^i_t$ is non-decreasing,  the probability of extinction is zero: $$\mathbb{P}(N_{t} = 0 \textrm{ for some $t$} \ \vert \ N_{0} = \mathbf e_{i})=0.$$ Thus we find $\mathbb{P}(W = 0 \ \vert \ X_{0} = \mathbf e_{i})=0$, implying $\mathbb{P}(W = 0 )=0$ regardless of $\boldsymbol \pi$.
Then by the continuous mapping theorem, $$R_{t}^{i} = \frac{N_{t}^{i}}{\sum_{i}N_{t}^{i}} = \frac{N_{t}^{i}e^{-\omega t}}{\sum_{i}N_{t}^{i}e^{-\omega t}} \xrightarrow{a.s.} \frac{Wu_{i}}{W} = u_{i}$$ for each $i$.
\end{proof}
  
\begin{rmk}
In studying diversification models with a single type but time-dependent rates of speciation and extinction, it is common to consider the random function giving the the number of lineages through time in a tree. This loses no information on parameters from the full tree, as each change in its value (speciation or extinction) is equally likely to have occurred on any lineage, and the growth of this function is thus highly informative on
parameter values.  For the multitype pure-birth model, however, the function $\sum_i N^i_t$ should not capture all information in the tree, as speciation may not be equally likely on all lineages. Corollary \ref{cor4} indicates its growth is determined only by $\omega$, the largest eigenvalue of $A$.
\end{rmk}


\section{Identifiability of the MPBT model}\label{sec:main}

Using the distributions of edge lengths and relative frequencies of each type of edge in a tree at a given time found in Sections \ref{sec:edge} and \ref{sec:typecount}, we are ready to establish identifiability of the MPBT parameters. To do so, we consider an asymptotic joint distribution of the lengths of 3 edges around a common node in the tree (see Figure \ref{fig:f4}). We seek to show that from this distribution the model parameters $(\boldsymbol\lambda, S)$ can be determined, up to label swapping.

\begin{figure}
\begin{center}
\includegraphics[width=4in]{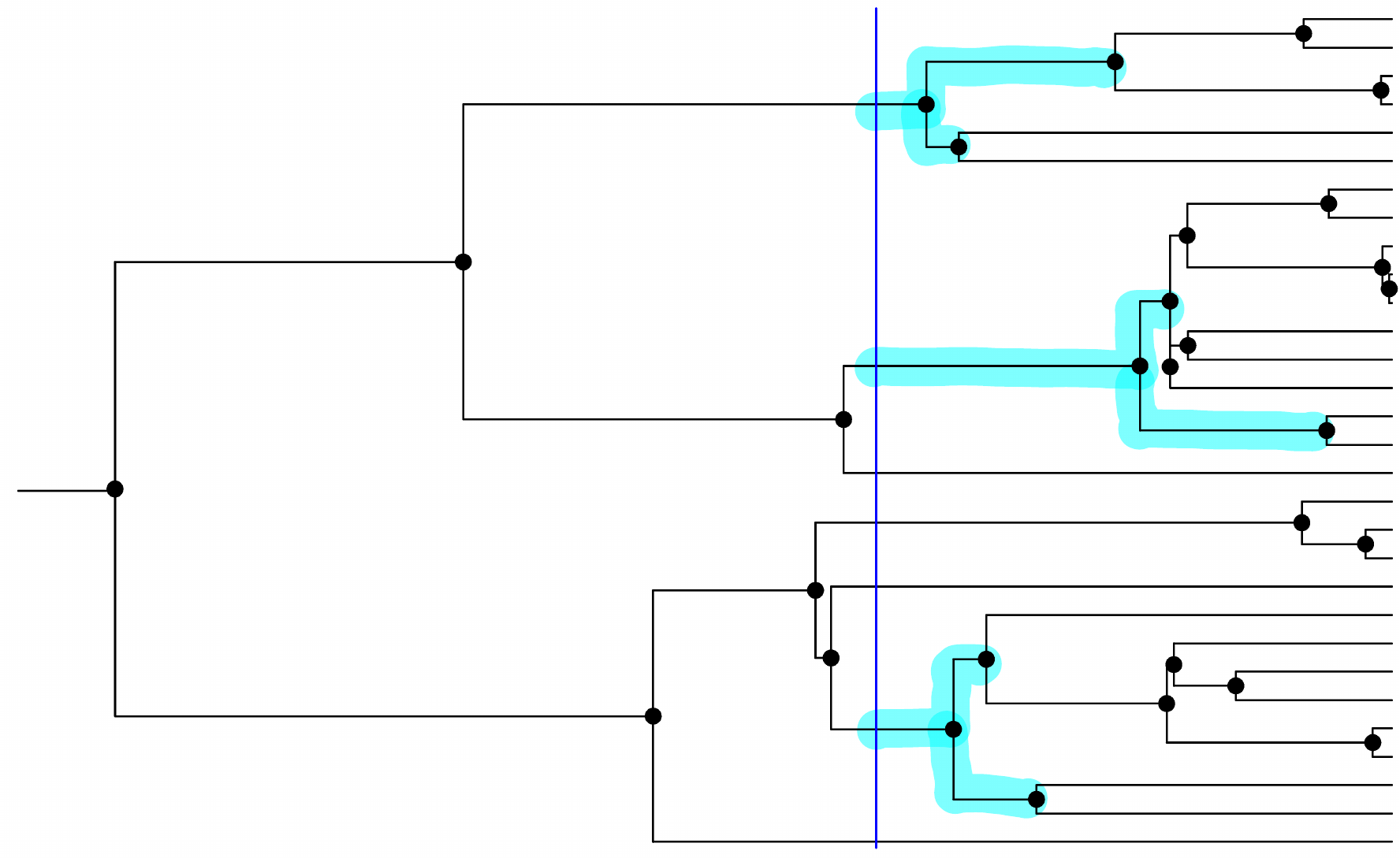}
\end{center}
\hskip 2.4in $t$ \hskip 1.3in$T$\hfill
\caption{The uncolored tree of Figure \ref{fig:f1}, of depth $T$, generated by the binary-type pure birth model. The blue line at $t$ determines several highlighted triples of edges whose lengths are possible draws from the probability distribution $G_t$ of Definition \ref{def:G} of Section \ref{sec:main}.}\label{fig:f4}
\end{figure}

Due to the conditional independence of the lengths of three edges sharing a common node, given that node's type, this distribution is a mixture of product distribution, with the mixing distribution and the components of the products closely related to distributions previously computed. This structure allows for the application of the following theorem, to obtain unmixed distributions of edge lengths conditioned on the type of the parental node. Thus even though we have no observation of type at any point in the tree, we can extract a distribution that is conditioned on type.

The following is a variant of Theorem 8 of  \cite{Allman2009}, with the hypotheses modified as discussed on p.~3116 of that paper.

\begin{theorem} \cite{Allman2009} \label{thm:mixprod}
For  $1\le i\le m$, let $$\mu_i=\prod_{k=1}^3 \mu_i^k$$ be a product of $3$ independent, absolutely continuous distributions $\mu_i^k$ on $\mathbb R$. 
With $\pi_i>0$, let $(\pi_1,\pi_2, \dots,\pi_m)$ be a distribution on $[m]$. For each $k$, suppose the
set of distributions $\{\mu_i^k\}_{i=1}^m$ has the property that every subset of  $r_k$ elements is linearly independent, and that $$r_1+r_2+r_3 \ge 2m+2.$$
Then, up to label swapping in $i$, the $\mu_i^k$ and $\pi_i$ are determined by the mixture distribution
$$P = \sum_{i = 1}^{r}\pi_i\mu_{i} = \sum_{i = 1}^{r}\pi_{i}\prod_{k = 1}^{n}\mu_i^k.$$ 
More precisely, $P$ determines distributions $\nu_i^k$ and $(p_1,p_2,\dots, p_m)$ such that for some permutation $\sigma$ of the set $[m]$,
$$\mu_i^k=\nu_{\sigma(i)}^k \text{ and } \pi_i=p_{\sigma(i)}.$$
\end{theorem}

To apply this theorem, we make a further technical assumption, denoting the vector of $1$s by $\mathbf 1$. 

\begin{assump}\label{a:nonsing}
Parameters $(\boldsymbol \lambda, S)$ are such that the  $m\times m$ matrix $$M=M(\boldsymbol \lambda,S)=\begin{pmatrix} \mathbf 1 &\  {U}\mathbf 1&\  {U}^2\mathbf 1&\ \dots&\  {U}^{m-1}\mathbf 1\end{pmatrix}$$  is non-singular.
\end{assump}

While the role of this assumption in our arguments will  be clear in our proofs of Lemma \ref{lem:indep} and Theorem \ref{thm:main} below, to understand its implications concretely, consider first the case $m=2$. Then $$U=\begin{pmatrix} -s_{12}-\lambda_1&\  s_{12}\\s_{21} &\ -s_{21}-\lambda_2\end{pmatrix},$$ so
$$M=\begin{pmatrix} 1&\  -\lambda_1\\1&\  -\lambda_2\end{pmatrix}.$$ 
The non-singularity of $M$ thus is equivalent to $\lambda_1\ne \lambda_2.$ That these speciation rates would need to be different for parameters to be identifiable is intuitively clear, since otherwise type changes governed by $S$ would have no impact on the structure of the uncolored tree.

For general $m$, Assumption \ref{a:nonsing} is equivalent to the non-vanishing of $\det M$, a degree $\sum_{i=1}^{m-1} i=\binom m2$ polynomial in the $m^2$ independent entries of $\boldsymbol \lambda, S$. Its non-vanishing thus excludes an algebraic variety of codimension 1, a set of Lebesque measure 0 in the unrestricted parameter space. An explicit calculation in the $m=3$ case shows the polynomial to be an irreducible polynomial in the $\lambda_i$ and $s_{ij}$, $i\ne j$. 

The non-vanishing of $\det M$ always requires that the vector $\boldsymbol \lambda=-U\mathbf 1$ not be a multiple of $\mathbf 1$ (so that the first two columns of $M$ are linearly independent), and hence that not all $\lambda_i$ are the same. However, the additional restrictions it imposes on the parameters are more opaque to intuition without considering special cases.  

For instance, when $m=3$, if all the $s_{ij}$ are equal, so the type switching behavior is identical for all types, the polynomial simplifies considerably, and factors as
$$(\lambda_1-\lambda_2)(\lambda_2-\lambda_3) (\lambda_3-\lambda_1).$$
Non vanishing of the polynomial, then requires that the three $\lambda_i$ be distinct, as one would expect is needed for identifiability, for otherwise several types would behave identically. However, for other choices of the $s_{ij}$, two of the $\lambda_i$ can be equal without the polynomial vanishing.

\medskip

Next, we define  the joint edge length distribution for several edges of a tree.

\begin{definition} \label{def:G} For some $t< T$, consider  the following three random variables: Sample an (uncolored) tree of depth $T$ under the MPBT model. From among the edges of the tree existing at time $t$ choose one uniformly at random. Then with $t_b\in(t,T)$, the time at which that edge speciates,  let $\ell^0_{t}=t_b-t$ denote the time interval until it speciates, and let $\ell^1_{t}$ and $\ell^{2}_{t}$, respectively denote the lengths of the immediate descendent edges (where the edges are designated 1,2  uniformly at random). Then the joint distribution of these three variables $\ell ^0_{t}$, $\ell^1_{t}$, and $\ell^2_{t}$ is
 \begin{equation*}
G_{T,t}(\tau_{0}, \tau_{1}, \tau_{2}) := \mathbb{P}\left (\ell^{0}_{t} \le \tau_{0}, \ell^{1}_{t} \le \tau_{1}, \ell^{2}_{t} \le \tau_{2}  \mid  \ell^{1}_{t}, \ell ^{2}_{t} < T-t-\ell^{0}_{t}  \right ).
\end{equation*}
We call $G_{T,t}$ the \emph{joint distribution of edge lengths around a node}.
\end{definition}

The three edge lengths used in the definition of $G_{T,t}$ are depicted in Figure \ref{fig:f4}, for $t=T/2$.  The conditioning in the definition of $G_{T,t}$ ensures it only considers edges in which the edge process has led to speciation, that is, the edge processes for the parental  and child edges are not truncated.

\begin{lemma} 
\label{lem:GGinf} Under Assumptions \ref{a:lambdapos} and \ref{a:spos},
as $T \to \infty$, the joint distribution $G_{T,T/2}$  at time $T/2$ of edge lengths around a node on a tree of depth $T$ converges to 
\begin{equation}
G_{\infty} = \sum_{i}\sum_{j} u_i P_{i_{-}, j_{+}}(\tau_0)F_{j}(\tau_1)F_{j}(\tau_2),\label{eq:Ginf}
\end{equation}
where $F_{j}$, $P_{i_{-}, j_{+}}$, and $u_{i}$ are defined in Propositions  \ref{prop:condT}, \ref{prop:asycondprob}, and Lemma \ref{lem3}, respectively.
\end{lemma}

\begin{proof}

Note that the event $E$ which
is conditioned upon in the definition of $G_{T,T/2}$ excludes edge lengths resulting from truncated edge processes, so that all edge lengths under consideration are in fact speciation times $\mathcal T$. Thus
\begin{align*}
\lim_{T \to \infty}G_{T,T/2}(\tau_{0}, &\tau_{1}, \tau_{2})\\
 &= \lim_{T \to \infty}\mathbb{P}(\mathcal T^0_{T/2} \le \tau_{0}, \mathcal T^{1}_{T/2} \le \tau_{1}, \mathcal T^{2}_{T/2} \le \tau_{2} \mid \mathcal T^{1}_{T/2}, \mathcal T ^{2}_{T/2} < T/2-\mathcal T^{0}_{T/2}
)\\
 &= \lim_{T \to \infty}\left ( \mathbb{P}(\mathcal T^0_{T/2} \le \tau_{0}, \mathcal T^{1}_{T/2} \le \tau_{1}, \mathcal T^{2}_{T/2} \le \tau_{2})+\epsilon_T(\tau_0,\tau_1,\tau_2)\right),
\end{align*} 
where the function $\epsilon_T$ is the difference of the conditional and non-conditional probabilities above. But since the probability of $E\to 1$ as $T\to \infty$, it  follows that $\epsilon_T\to 0$.
We henceforth focus on $\mathbb{P}(\mathcal T^{0}_{T/2} \le \tau_{0}, \mathcal T^{1}_{T/2} \le \tau_{1}, \mathcal T^{2}_{T/2} \le \tau_{2} )$ rather than $G_{T,T/2}$.

Letting $A_{i}$ denote the event that the uniformly-at-random chosen edge is of type $i$ at time $\frac{T}{2}$ and $B_{j}$ denote the event that that edge speciates in color $j$, and recalling that edge processes around a node are independent when conditioned on the type of that node, we have

\begin{align*}
\mathbb{P}(\mathcal T^{0}_{T/2} \le \tau_{0}, &\mathcal T^{1}_{T/2} \le \tau_{1}, \mathcal T^{2}_{T/2} \le \tau_{2})\\& =
\sum_{i}\sum_{j}\mathbb{P}(\mathcal T^{0}_{T/2} \le \tau_{0}, \mathcal T^{1}_{T/2} \le \tau_{1}, \mathcal T^{2}_{T/2} \le \tau_{2} \ \vert \ A_{i}, B_{j})\mathbb{P}(A_{i}, B_{j}) \\
&=\sum_{i}\sum_{j}\mathbb{P}(\mathcal T^{0}_{T/2} \le \tau_{0}, B_{j} \ \vert \ A_{i})\mathbb{P}(\mathcal T^{1}_{T/2} \le \tau_{1} \ \vert \ B_{j})\mathbb{P}(\mathcal T^{2}_{T/2} \le \tau_{2} \ \vert \ B_{j})\mathbb{P}(A_{i}) \\
&=\sum_{i}\sum_{j}P_{i_{-}, j_{+}}(\tau_{0})F_{j}(\tau_{1})F_{j}(\tau_{2})\mathbb{P}(A_{i}). 
\end{align*}
In this last expression, the only dependence on $T$ is in $\mathbb P(A_i)$. But by  Corollary \ref{cor4}, $\mathbb{P}(A_{i})=\mathbb E [R^i_{T/2}] \to u_i$  as $T\to \infty$, yielding equation
\eqref{eq:Ginf}.
\end{proof}

\begin{rmk}
While the specific time $T/2$ is used in this Lemma, our arguments would be essentially unchanged if this were replaced by any function $f(T)$  with $f(T)$ and $T-f(T)\to \infty$ as $T\to \infty$.
\end{rmk}

This immediately gives that $G_\infty$ is a finite mixture of product distributions.

\begin{cor} 
\label{cor:Ginf}
The asymptotic joint distribution of edge lengths around a node, $G_{\infty}$ can be expressed as a $m$-component mixture of products of $3$ univariate distributions:  $$G_{\infty} = \sum_{j = 1}^{2}\pi_{j}\prod_{k = 1}^{3}\mu^{k}_{j},$$ where $\pi_{j} = \sum_{i}P_{i_{-}, j_{+}}u_{i}$, $\mu_{j}^{1} = \frac{\sum_{i}P_{i_{-}, j_{+}}(\tau)u_{i}}{\sum_{i}P_{i_{-}, j_{+}}u_{i}}$, $\mu_{j}^{2} = \mu_{j}^{3} = F_{j}(\tau)$, and $P_{i_{-}, j_{+}}$ is as defined in Proposition \ref{prop:asycondprob}.
\end{cor}

In order to apply Theorem \ref{thm:mixprod} to $G_\infty$, we need to verify that some of the univariate distributions in its decomposition above are linearly independent. To do so, the following lemma is needed.

We now introduce an additional assumption, which holds for generic parameters.

\begin{assump} \label{a:notequal} The speciation parameters satisfy $\lambda_i\ne \lambda_j$ for all $i\ne j$.
\end{assump}

\begin{lemma} 
\label{lem:indep} Suppose Assumption \ref{a:lambdapos},\ref{a:spos}, \ref{a:nonsing}, and \ref{a:notequal} hold, and consider the sets of univariate distributions $\{\mu^{k}_{j}\}_{j=1}
^m$  defined in Corollary $\ref{cor:Ginf}$.
For $k = 1$, every pair of  functions in this set is linearly independent, while for $k=2,3$ the full set is linearly independent.
\end{lemma}

\begin{proof} 

Since $\{\mu^{2}_{j}\}_{j} = \{\mu^{3}_{j}\}_{j}$, we need only consider the cases $k = 1, 2$. 

Consider first the case $k = 2$. Consider the vector $F$ of functions $\mu^2_j=F_j$. Then by Proposition \ref{prop:condT},
$$F= \mathbf{1} - \exp({U\tau})\mathbf{1}.$$
Suppose $\mathbf c^T F=0$ for some vector $\mathbf c$.
Since $\frac {d^n} {d\tau^n} F (0)=-U^n \mathbf{1}$, it follows that $\mathbf c^TM=\mathbf 0$ where $M$ is defined in Assumption \ref{a:nonsing}. Since $M$ is non-singular, $\mathbf c=\mathbf 0$, so  the entries of $F$ are independent.

For $k = 1$,
it is enough to show the independence of each pair of functions $$\nu_j(\tau)=(\sum_{i}P_{i_{-}, j_{+}}  u_{i}) \mu^1_j =\sum_{i}P_{i_{-}, j_{+}}(\tau)u_{i}.$$ 
From Lemma \ref{lem:edgeTrans} the vector $G$ of all $\nu_j$ is given by
$$G(\tau)^T =\mathbf u^T \sum_{n=1}^\infty \frac 1{n!} U^{n-1}D\tau^n.$$
Suppose $G(\tau)^T\mathbf c =0$ for some vector $\mathbf c$.
Since $\frac {d^n} {d\tau^n} G(0)^T=\mathbf u^T U^{n-1}D$, it follows that
$$ \mathbf u^T U^{n-1}D \mathbf c=0 \text{ for } n\ge 1. $$
In particular, for $n=1$ we find $\mathbf u^TD\mathbf c=0$. For $n=2$, since 
$U=A-2D$ and $\mathbf u^TA= \omega \mathbf u^T$, we have $$\mathbf u^T UD\mathbf c=\mathbf u^T(\omega I-2D)D\mathbf c=0.$$
To show every pair of the $\nu_j$s is independent, consider $\mathbf c$ all of whose entries except possibly two are zero. Without loss of generality suppose the exceptions are $c_1,c_2$. Then the $n=1,2$ equations become
$$\begin{pmatrix}
u_1 \lambda_1 & u_2 \lambda _2\\
u_1(\omega-2\lambda_1)\lambda_1 & u_2(\omega-2\lambda_2)\lambda_2
\end{pmatrix}
 \begin{pmatrix}c_1\\c_2\end{pmatrix} =\mathbf 0
$$
Using $u_1,u_2,\lambda_1,\lambda_2 >0, \lambda_1\ne \lambda_2$, computing the determinant of this matrix shows it is non-singular, and hence $c_1=c_2=0$.

\end{proof}

We now arrive at our main result.

\begin{theorem}
\label{thm:main} Under the explicit generic  Assumptions  \ref{a:lambdapos},  \ref{a:spos} \ref{a:nonsing},  and \ref{a:notequal},
the parameters $(\boldsymbol \lambda, S)$ of the uncolored Multitype Pure-birth Tree model are identifiable up to label swapping from the asymptotic distribution $G_\infty$ of edge lengths around a node.
\end{theorem}

\begin{proof} Suppose two parameter choices, $(\boldsymbol \pi, \boldsymbol \lambda, S)$ and $(\boldsymbol \pi^*,\boldsymbol \lambda^*, S^*)$,  induce the same asymptotic distribution $G_\infty$.  Denoting the various distributions of conditional branching times, asymptotic transition probabilities, eigenvectors of matrices, etc. associated to parameters  $(\boldsymbol \pi, \boldsymbol \lambda, S)$ as earlier in this work, we use the same notation with a ``$*$'' appended to denote the corresponding entities associated to parameters $(\boldsymbol \pi^*,\boldsymbol \lambda^*, S^*)$.

By Theorem  \ref{thm:mixprod}, Corollary \ref{cor:Ginf}, and Lemma \ref{lem:indep} the distributions $\pi_{i}, \mu_{i}^{k}$, for $1 \le i \le r$, $1 \le k \le 3$ are determined from $G_{\infty} = G^*_{\infty}$, up to label swapping in $i$. Thus $F^*_{i}(\tau) = F_{\sigma(i)}(\tau)$  for some permutation $\sigma$. 

Using Proposition \ref{prop:condT} the equations $F^*_{i}(\tau) = F_{\sigma(i)}(\tau)$ for all $j$ can be represented  in matrix form as 
\begin{equation}
\mathbf 1 - e^{{U}^*\tau}\mathbf 1 = \Sigma(\mathbf 1 - e^{{U}\tau}\mathbf 1), 
\end{equation} 
where $\Sigma$ is the permutation matrix representing $\sigma$. Equating coefficients of the MacLauren series yields for $n=1,2,3,\dots$ that
\begin{equation} ({U}^*)^n\mathbf 1=\Sigma {U}^n\mathbf 1.\label{eq:SS}\end{equation}

Using equation \eqref{eq:SS} and the definition of $M,M^*$ in Assumption \ref{a:nonsing} shows
\begin{equation}M^*=\Sigma M.\label{eq:S1S1}\end{equation}
Equation \eqref{eq:SS} further implies
\begin{align*}{U}^* M^*&=\begin{pmatrix} {U}^*\mathbf 1 &\  ({U}^*)^2\mathbf 1&\  ({U}^*)^3\mathbf 1&\ \dots&\  ({U}^*)^{r}\mathbf 1\end{pmatrix}\\
&=\begin{pmatrix} \Sigma {U}\mathbf 1 &\  \Sigma {U}^2\mathbf 1&\  \Sigma {U}^3\mathbf 1&\ \dots&\  \Sigma {U}^{r}\mathbf 1\end{pmatrix}\\
&=\Sigma {U} M.
\end{align*}
Using equation \eqref{eq:S1S1} then yields
$$  {U}^* \Sigma M=\Sigma {U} M,$$
and since $M$ is non-singular,
$${U}^*\Sigma =\Sigma {U}.$$
  
 Since $U=S-D$ and each row of $S$ adds to 0, multiplying the last equation by $\mathbf 1$ on the right
 gives $\boldsymbol \lambda^*=\Sigma \boldsymbol \lambda.$ Since this implies $D^*\Sigma=\Sigma D$, it follows
 that $S^*=\Sigma S \Sigma^T$ as well.
 Thus the parameters differ only up to label swapping.
  \end{proof}

\begin{rmk}\label{rmk:1tree}
Theorem \ref{thm:main} establishes that an asymptotic distribution, as tree depth $\to \infty$ associated to the MPBT model yields parameter identifiability. This suggests that with a sample of many trees of arbitrarily large size, there is potential for statistically consistent inference, where ``consistency'' would mean as both the number of trees and the tree depth go to infinity. However, this is not the framework in which data analysis with this model is performed, since while a tree may be large, only one tree observation is available \cite{Maddison2007}.

Fortunately, a minor modification to the proofs above again yields identifiability of parameters from an asymptotic distribution derived from a single observation, as the depth of the tree goes to infinity. Indeed, modify Definition \ref{def:G}  so that $G_t$ is the distribution of edge lengths around a node from single growing tree.  The proof of Lemma \ref{lem:GGinf}, then, is modified only in its last line, as $\mathbb P(A_i)=R^i_{T/2}$, a random variable rather than its expected value. Nonetheless, by Corollary \ref{cor4}, 
we again find $\mathbb P(A_i)\to u_i$, so the conclusion is unchanged.
\end{rmk}

\section{Discussion}\label{sec:discuss}

Theorem \ref{thm:main}, and Remark \ref{rmk:1tree}, show that parameters $(\boldsymbol \lambda, S)$ of the MPBT model can be identified from an asymptotic distribution as the tree depth grows, whether or not the number of sampled trees grows.
Although this is not sufficient to conclude that estimation of parameters by maximum likelihood (ML) from  a single tree, as suggested in \cite{Maddison2007}, is statistically consistent, it does at least indicate that is a possibility. A similar question on ML inference of parameters for a hidden Markov model from a single sequence of observations was addressed by \cite{Leroux1992}, with the consistency of ML estimation established as the sequence length goes to infinity.

For applications, it would be highly desirable to extend our identifiability result to a model incorporating constant extinction rates for each type.
In most biological settings, the obtainable ``data," however is not the tree with edges stopping at extinction events, but rather the pruned tree in which all edges with no extant descendants are removed.  

For a single type, parameter identifiability of a model with pruning was essentially considered in \cite{Nee1994}, where it was shown that the lineages-through-time function's rate of change allowed the speciation and extinction rates to be determined, by separately considering the time regimes much earlier than the tree tips, and near the tree tips. An analysis combining the insight from \cite{Nee1994} with the 
mixture distribution framework used in this work might be successful in showing parameters can be recovered from a single large tree observation for the multitype model.

Another interesting identifiability question for multitype tree models concerns what information on parameters is contained in the tree topology alone, or from weaker metric information than precise branch lengths. While our analysis depends heavily on metric features of the tree, that of \cite{Popovic2016} required no metric information. However,  it did use type observations at the tips of the tree, and at their parental nodes. While types at tree tips may be observed in some biological studies, types of the parental nodes are generally not observable, as data is generally collected only from the taxa extant at the present. Even if ancient DNA or other trait data from earlier times is available, it is unlikely to be from the time of the last speciation. 

\begin{funding}
ESA and JAR  were supported in part by NSF Grant DMS-2051760.
\end{funding}

\bibliographystyle{imsart-number}
\bibliography{Diversification}                         

\begin{thebibliography}{18}

\bibitem{Allman2009}
\begin{barticle}[author]
\bauthor{\bsnm{Allman},~\bfnm{Elizabeth~S.}\binits{E.~S.}},
  \bauthor{\bsnm{Matias},~\bfnm{Catherine}\binits{C.}} \AND
  \bauthor{\bsnm{Rhodes},~\bfnm{John~A.}\binits{J.~A.}}
(\byear{2009}).
\btitle{Identifiability of Parameters in Latent Structure Models With Many
  Observed Variables}.
\bjournal{Ann. Statist.}
\bvolume{37}
\bpages{3099--3132}.
\bdoi{10.1214/09-AOS689}
\end{barticle}
\endbibitem

\bibitem{Athreya1968}
\begin{barticle}[author]
\bauthor{\bsnm{Athreya},~\bfnm{Krishna~Balasundaram}\binits{K.~B.}}
(\byear{1968}).
\btitle{Some Results on Multitype Continuous Time Markov Branching Processes}.
\bjournal{The Annals of Mathematical Statistics}
\bvolume{39}
\bpages{347 -- 357}.
\bdoi{10.1214/aoms/1177698395}
\end{barticle}
\endbibitem

\bibitem{Barido-Sottani2020}
\begin{barticle}[author]
\bauthor{\bsnm{Barido-Sottani},~\bfnm{Jo{\"e}lle}\binits{J.}},
  \bauthor{\bsnm{Vaughan},~\bfnm{Timothy~G}\binits{T.~G.}} \AND
  \bauthor{\bsnm{Stadler},~\bfnm{Tanja}\binits{T.}}
(\byear{2020}).
\btitle{A Multitype Birth--Death Model for {B}ayesian Inference of
  Lineage-Specific Birth and Death Rates}.
\bjournal{Systematic Biology}
\bvolume{69}
\bpages{973-986}.
\bdoi{10.1093/sysbio/syaa016}
\end{barticle}
\endbibitem

\bibitem{FitzJohn2012}
\begin{barticle}[author]
\bauthor{\bsnm{FitzJohn},~\bfnm{Richard~G.}\binits{R.~G.}}
(\byear{2012}).
\btitle{Diversitree: comparative phylogenetic analyses of diversification in
  {R}}.
\bjournal{Methods in Ecology and Evolution}
\bvolume{3}
\bpages{1084-1092}.
\end{barticle}
\endbibitem

\bibitem{Harvey1994}
\begin{barticle}[author]
\bauthor{\bsnm{Harvey},~\bfnm{Paul~H.}\binits{P.~H.}},
  \bauthor{\bsnm{May},~\bfnm{Robert~M.}\binits{R.~M.}} \AND
  \bauthor{\bsnm{Nee},~\bfnm{Sean}\binits{S.}}
(\byear{1994}).
\btitle{Phylogenies Without Fossils}.
\bjournal{Evolution}
\bvolume{48}
\bpages{523--529}.
\end{barticle}
\endbibitem

\bibitem{Horn2012}
\begin{bbook}[author]
\bauthor{\bsnm{Horn},~\bfnm{R.~A.}\binits{R.~A.}} \AND
  \bauthor{\bsnm{Johnson},~\bfnm{C.~R.}\binits{C.~R.}}
(\byear{2012}).
\btitle{Matrix Analysis}.
\bpublisher{Cambridge University Press}.
\end{bbook}
\endbibitem

\bibitem{Kendall1948}
\begin{barticle}[author]
\bauthor{\bsnm{Kendall},~\bfnm{David~G.}\binits{D.~G.}}
(\byear{1948}).
\btitle{On the Generalized ``Birth-And-Death'' Process}.
\bjournal{The Annals of Mathematical Statistics}
\bvolume{19}
\bpages{1 -- 15}.
\bdoi{10.1214/aoms/1177730285}
\end{barticle}
\endbibitem

\bibitem{Leroux1992}
\begin{barticle}[author]
\bauthor{\bsnm{Leroux},~\bfnm{Brian~G.}\binits{B.~G.}}
(\byear{1992}).
\btitle{Maximum-likelihood estimation for hidden {M}arkov models}.
\bjournal{Stochastic Processes and their Applications}
\bvolume{40}
\bpages{127-143}.
\bdoi{https://doi.org/10.1016/0304-4149(92)90141-C}
\end{barticle}
\endbibitem

\bibitem{Louca2020}
\begin{barticle}[author]
\bauthor{\bsnm{Louca},~\bfnm{Stilianos}\binits{S.}} \AND
  \bauthor{\bsnm{Pennell},~\bfnm{Matthew~W.}\binits{M.~W.}}
(\byear{2020}).
\btitle{Extant Timetrees Are Consistent With a Myriad of Diversification
  Histories}.
\bjournal{Nature}
\bvolume{580}
\bpages{502-505}.
\bdoi{10.1038/s41586-020-2176-1}
\end{barticle}
\endbibitem

\bibitem{Maddison2007}
\begin{barticle}[author]
\bauthor{\bsnm{Maddison},~\bfnm{Wayne~P.}\binits{W.~P.}},
  \bauthor{\bsnm{Midford},~\bfnm{Peter~E.}\binits{P.~E.}} \AND
  \bauthor{\bsnm{Otto},~\bfnm{Sarah~P.}\binits{S.~P.}}
(\byear{2007}).
\btitle{Estimating a Binary Character's Effect on Speciation and Extinction}.
\bjournal{Systematic Biology}
\bvolume{56}
\bpages{701-710}.
\bdoi{10.1080/10635150701607033}
\end{barticle}
\endbibitem

\bibitem{Maliet2019}
\begin{barticle}[author]
\bauthor{\bsnm{Maliet},~\bfnm{Odile}\binits{O.}},
  \bauthor{\bsnm{Hartig},~\bfnm{Florian}\binits{F.}} \AND
  \bauthor{\bsnm{Morlon},~\bfnm{Helene}\binits{H.}}
(\byear{2019}).
\btitle{A model with many small shifts for estimating species-specific
  diversification rates}.
\bjournal{Nature Ecology \& Evolution}
\bvolume{3}
\bpages{1086--1092}.
\bdoi{10.1038/s41559-019-0908-0}
\end{barticle}
\endbibitem

\bibitem{Nee1994}
\begin{barticle}[author]
\bauthor{\bsnm{Nee},~\bfnm{Sean}\binits{S.}},
  \bauthor{\bsnm{May},~\bfnm{Robert~Mccredie}\binits{R.~M.}} \AND
  \bauthor{\bsnm{Harvey},~\bfnm{Paul~H.}\binits{P.~H.}}
(\byear{1994}).
\btitle{The Reconstructed Evolutionary Process}.
\bjournal{Philosophical Transactions of the Royal Society of London. Series B:
  Biological Sciences}
\bvolume{344}
\bpages{305-311}.
\bdoi{10.1098/rstb.1994.0068}
\end{barticle}
\endbibitem

\bibitem{Popovic2016}
\begin{barticle}[author]
\bauthor{\bsnm{Popovic},~\bfnm{Lea}\binits{L.}} \AND
  \bauthor{\bsnm{Rivas},~\bfnm{Mariolys}\binits{M.}}
(\byear{2016}).
\btitle{Topology and inference for {Y}ule trees with multiple states}.
\bjournal{J. Math. Biol.}
\bvolume{73}
\bpages{1251--1291}.
\bdoi{10.1007/s00285-016-0992-6}
\end{barticle}
\endbibitem

\bibitem{RaboskyGoldberg2015}
\begin{barticle}[author]
\bauthor{\bsnm{Rabosky},~\bfnm{Daniel~L.}\binits{D.~L.}} \AND
  \bauthor{\bsnm{Goldberg},~\bfnm{Emma~E.}\binits{E.~E.}}
(\byear{2015}).
\btitle{Model Inadequacy and Mistaken Inferences of Trait-Dependent
  Speciation}.
\bjournal{Systematic Biology}
\bvolume{64}
\bpages{340-355}.
\bdoi{10.1093/sysbio/syu131}
\end{barticle}
\endbibitem

\bibitem{RasStadler2019}
\begin{barticle}[author]
\bauthor{\bsnm{Rasmussen},~\bfnm{David~A}\binits{D.~A.}} \AND
  \bauthor{\bsnm{Stadler},~\bfnm{Tanja}\binits{T.}}
(\byear{2019}).
\btitle{Coupling adaptive molecular evolution to phylodynamics using
  fitness-dependent birth-death models}.
\bjournal{eLife}
\bvolume{8}
\bpages{e45562}.
\bdoi{10.7554/eLife.45562}
\end{barticle}
\endbibitem

\bibitem{Stadler2013}
\begin{barticle}[author]
\bauthor{\bsnm{Stadler},~\bfnm{Tanya}\binits{T.}}
(\byear{2013}).
\btitle{Recovering Speciation and Extinction Dynamics Based on Phylogenies}.
\bjournal{Journal of Evolutionary Biology}
\bvolume{26}
\bpages{1203-1219}.
\bdoi{https://doi.org/10.1111/jeb.12139}
\end{barticle}
\endbibitem

\bibitem{Stadler2019}
\begin{barticle}[author]
\bauthor{\bsnm{Stadler},~\bfnm{Tanja}\binits{T.}}
(\byear{2019}).
\btitle{Species-specific diversification}.
\bjournal{Nature Ecology \& Evolution}
\bvolume{3}
\bpages{1003--1004}.
\bdoi{10.1038/s41559-019-0923-1}
\end{barticle}
\endbibitem

\bibitem{Yule1924}
\begin{barticle}[author]
\bauthor{\bsnm{Yule},~\bfnm{George~Udny}\binits{G.~U.}}
(\byear{1925}).
\btitle{A Mathematical Theory of Evolution, Based on the Conclusions of {Dr. J.
  C. Willis, F.R.S.}}
\bjournal{Philosophical Transactions of the Royal Society of London. Series B,
  Containing Papers of a Biological Character}
\bvolume{213}
\bpages{21-87}.
\bdoi{10.1098/rstb.1925.0002}
\end{barticle}
\endbibitem

\end{thebibliography}

\end{document}